\documentclass[12pt]{article}
\usepackage[english]{babel}

\usepackage{upref,amsfonts,amsxtra,a4wide,latexsym,enumerate}
\usepackage{amssymb,a4wide,epsf,mathrsfs}
\usepackage{amsthm}
\usepackage{amsmath}
\usepackage{lastpage}
\usepackage[all]{xy}
\usepackage{cleveref}
\usepackage{pdfpages}
\usepackage{enumitem}
\usepackage{graphicx}
\usepackage{float}

\usepackage{tikz}

\newcommand{\Z}{{\mathbb Z}}
\newcommand{\N}{{\mathbb N}}

\renewcommand{\phi}{{\varphi}}

\theoremstyle{plain}
\newtheorem{theorem}{Theorem}
\newtheorem{proposition}[theorem]{Proposition}
\newtheorem{lemma}[theorem]{Lemma}
\newtheorem{corollary}[theorem]{Corollary}
\theoremstyle{definition}
\newtheorem{definition}[theorem]{Definition}

\newtheorem{conjecture}[theorem]{Conjecture}

\crefname{definition}{Definition}{Definitions}
\Crefname{definition}{Definition}{Definitions}
\crefname{theorem}{Theorem}{Theorems}
\Crefname{theorem}{Theorem}{Theorems}
\crefname{lemma}{Lemma}{Lemmas}
\Crefname{lemma}{Lemma}{Lemmas}
\crefname{corollary}{Corollary}{Corollaries}
\Crefname{corollary}{Corollary}{Corollaries}
\crefname{proposition}{Proposition}{Propositions}
\Crefname{proposition}{Proposition}{Propositions}
\crefname{problem}{Problem}{Problems}
\Crefname{problem}{Problem}{Problems}
\crefname{conjecture}{Conjecture}{Conjectures}
\Crefname{conjecture}{Conjecture}{Conjectures}

\begin{document}

\title{Group connectivity and group coloring: small groups versus large groups}

\author{Rikke Langhede and Carsten Thomassen \thanks{Supported by Independent Research Fund Denmark, 8021-002498 AlgoGraph.}\\Department of Applied Mathematics and Computer Science,
\\Technical University of Denmark, DK-2800 Lyngby, Denmark}

\maketitle

\begin{abstract}
  A well-known result of Tutte says that if $\Gamma$ is an Abelian group and $G$ is a graph having a nowhere-zero $\Gamma$-flow, then $G$ has a nowhere-zero $\Gamma'$-flow for each Abelian group $\Gamma'$ whose order is at least the order of $\Gamma$. Jaeger, Linial, Payan, and Tarsi observed that this does not extend to their more general concept of group connectivity. Motivated by this we define $g(k)$ as the least number such that, if $G$ is $\Gamma$-connected for some Abelian group $\Gamma$ of order $k$, then $G$ is also $\Gamma'$-connected for every Abelian group $\Gamma'$ of order $|\Gamma'| \geq g(k)$. We prove that $g(k)$ exists and satisfies for infinitely many $k$,

\begin{align*}
(2-o(1)) k < g(k) \leq 8k^3+1.
\end{align*}

The upper bound holds for all $k$. Analogously, we define $h(k)$ as the least number such that, if $G$ is $\Gamma$-colorable for some Abelian group $\Gamma$ of order $k$, then $G$ is also $\Gamma'$-colorable for every Abelian group $\Gamma'$ of order $|\Gamma'| \geq h(k)$. Then $h(k)$ exists and satisfies for infinitely many $k$,

\begin{align*}
(2-o(1)) k < h(k) < (2+o(1))k \ln(k).
\end{align*}

The upper bound (for all $k$) follows from a result of Kr\'al', Pangr\'ac, and Voss. The lower bound follows by duality from our lower bound on $g(k)$ as that bound is demonstrated by planar graphs.
\end{abstract}

\section{Introduction}

Tutte's 5-Flow Conjecture states that any 2-edge-connected graph has a nowhere-zero 5-flow (see e.g. \cite{bm, j3, jt, cq}). Seymour \cite{s} proved that every 2-edge-connected graph has a nowhere-zero 6-flow.

Jaeger et al. \cite{jlpt} introduced the concept of group-connectivity and proved that every 3-edge-connected graph is $\Gamma$-connected for any Abelian group $\Gamma$ of size $|\Gamma| \geq 6$. This extends the 6-flow theorem since that theorem is easily reduced to the 3-connected case as pointed out by Seymour \cite{s}.

Tutte \cite{to} (see also \cite{cq}) proved that if $\Gamma$ is an Abelian group and $G$ is a graph having a nowhere-zero $\Gamma$-flow, then $G$ has a nowhere-zero $\Gamma'$-flow for each Abelian group $\Gamma'$ whose order is at least the order of $\Gamma$. Jaeger et al. \cite{jlpt} observed that this does not extend to their more general concept of group connectivity. We prove that the statement becomes true if $\Gamma'$ is sufficiently large compared to $\Gamma$, as explained in the Abstract.

If $\Gamma$ is an Abelian group and $G$ is a planar 2-edge-connected graph, then $G$ is $\Gamma$-connected if and only if the dual graph $G$ is $\Gamma$-colorable. This duality was part of the motivation for these concepts \cite{jlpt}. This suggests the definition of $g(k)$, $h(k)$ in the Abstract, and we prove:

\begin{align}
(2-o(1)) k < g(k) \leq 8k^3+1,
\end{align}
and
\begin{align}
(2-o(1)) k < h(k) < (2+o(1))k \ln(k).
\end{align}

The lower bounds hold for infinitely many $k$, and we do not know if $g,h$ are monotone. The upper bounds hold for all $k$. The upper bound on $h(k)$ follows from a result of Kr\'al' et al. \cite{kpv}. The lower bound on $h(k)$ follows from the lower bound of $g(k)$ as that bound is demonstrated by planar graphs.

It was proved in \cite{li} that $g(3)=3$. We conjecture that $h(3)=3$ and prove that $h(3) \leq 5$.

The {\em group chromatic number} (see e.g. \cite{kpv, la}) $\chi_g(G)$ is the smallest number $k$ such that $G$ is $\Gamma$-colorable for every Abelian group $\Gamma$ of order at least $k$. We point out that there is another possible definition, namely the {\em weak group chromatic number} $\chi_{wg}(G)$ which is the smallest number $k$ such that $G$ is $\Gamma$-colorable for {\em some} Abelian group $\Gamma$ of order $k$. Clearly

\begin{align}
\chi(G) \leq \chi_{wg}(G) \leq \chi_g(G).
\end{align}

Thus our lower bound on $h(k)$ shows that $\chi_g(G)$ may be almost twice as big as $\chi_{wg}(G)$.

The {\em group connectivity number} (see e.g. \cite{lai}) $\Lambda_g(G)$ is the smallest number $k$ such that $G$ is $\Gamma$-connected for every Abelian group $\Gamma$ of order at least $k$.
Similarly, we define the {\em weak group connectivity number} $\Lambda_{wg}(G)$ which is the smallest number $k$ such that $G$ is $\Gamma$-connected for {\em some} Abelian group $\Gamma$ of order $k$.

Our lower bound on $g(k)$ shows that $\Lambda_g(G)$ may be almost twice as big as $\Lambda_{wg}(G)$.

\section{Group connectivity}

We use essentially the terminology and notation in \cite{bm}. We allow a graph to have multiple edges but no loops. If $v$ is a vertex in a directed graph, then $E^+(v)$ (respectively $E^-(v)$) denotes the set of edges going out from $v$ (respectively going into $v$).
Jaeger et al. \cite{jlpt} introduced the concept of group connectivity as follows.

\begin{definition} \label{def:delta}
Let $\Gamma$ be an Abelian group. The graph $G$ is said to be {\em $\Gamma$-connected} if the following holds: Given some orientation $D$ of $G$ and any function $\beta: V(G) \to \Gamma$ satisfying
$\sum_{v \in V(G)} \beta(v) = 0$,
there exists a function $f: E(G) \to \Gamma$ such that $\sum_{e \in E^+(v)} f(e) - \sum_{e \in E^-(v)} f(e) = \beta(v)$ for all $v \in V(G)$ and such that $f(e) \neq 0$ for all $e \in E(G)$.
\end{definition}

Note that the direction of an edge is not important. Indeed, we may replace "some orientation" by "any orientation" in the definition because we may replace $f(e)$ by $-f(e)$, if necessary.

Jaeger et al. \cite{jlpt} gave the following criterion for group connectivity in terms of forbidden flow values.

\begin{theorem} \label{thm:lower}
The graph $G$ is $\Gamma$-connected if and only if the following holds: Given any orientation $D$ of $G$ and any function $\varphi: E(G) \to \Gamma$, there exists a function $f: E(G) \to \Gamma$ which has $f(e) \neq \varphi(e)$ for all $e \in E(G)$ and $\sum_{e \in E^+(v)} f(e) - \sum_{e \in E^-(v)} f(e) = 0$ for all $v \in V(G)$.
\end{theorem}

The function $f$ in Theorem~\ref{thm:lower} is called a {\em flow}, and $f(e)$ is called a {\em flow value}. $\varphi(e)$ is called a {\em forbidden flow value}.

\section{A lower bound for $g$}

We shall use the following lemma which is an easy exercise. A proof can be found in \cite{jlpt}.

\begin{lemma} \label{lemma:subset}
Let $P$ be a cyclic group of prime order, let $S$ be a non-empty proper subset of $P$, and let $T$ be a subset of $P$ which contains at least two elements. Then $|S+T| > |S|$.
\end{lemma}

Let $q,k$ be natural numbers. We define $G_{q,k}$ as the graph consisting of two vertices $s,t$ connected by $q$ internally disjoint paths of length $k$.

\begin{theorem} \label{thm:prime}
Let $k$ be a natural number, and let $p$ be the smallest prime $> 2^{k-1}+1$. Then  $G_{q,2^{k-1}}$ is not $\Z_2^k$-connected for any odd $q$. $G_{q,2^{k-1}}$ is $\Z_p$-connected when $q \geq p$.
\end{theorem}

\begin{proof}
We first prove that $G=G_{q,2^{k-1}}$ is not $\Z_2^k$-connected.
By reversing directions, if necessary, we may assume that all edges are directed towards $t$.
Let $\varphi: E(G) \to \Z_2^k$ be a function such that each path forbids all the $2^{k-1}$ elements having an even number of 1's. As the flow values on a directed path has to be the same on all edges of the path, the forbidden values imply that the flow value of any of the paths must have an odd number of 1's.
Since $G$ consists of an odd number of paths which all have a flow with an odd number of 1's, the in-flow in $t$ (or, similarly, out-flow in $s$) can never sum to $0$. Thus, $G$ is not $\Z_2^k$-connected.

To prove that $G$ is $\Z_p$-connected we consider any function $\varphi: E(G) \to \Z_p$, and again, we let all edges be directed towards $t$.
The goal is to find a flow $f: E(G) \to \Z_p$ such that $f(e) \neq \varphi(e)$ for all $e \in E(G)$.
As noted above $f$ must have the same value on all edges of a path between $s$ and $t$, and as each path has length $2^{k-1}$, it has at least $p - 2^{k-1} \geq 3$ possible flow values of $\Z_p$.
It remains to check that we can choose the values of $f$ on each path such that the sum is 0 in $s$ and $t$. Given $q \geq p$ subsets of $\Z_p$ of size at least 2, it follows from Lemma~\ref{lemma:subset} that the sum of these contains all elements of $\Z_p$, in particular the sum contains 0.
Thus we can choose $f$ such that $f(e) \neq \varphi(e)$ for all $e \in E(G)$ and the sum in $s$ and $t$ is $0$.
\end{proof}

\begin{corollary} \label{cor:infinite}
Given any $\epsilon > 0$ there exists an infinite number of graphs which are $\Gamma$-connected for some group $\Gamma$ of prime order, but not $\Gamma'$-connected for some group $\Gamma'$ satisfying $|\Gamma'| = (2-\epsilon)|\Gamma|$.
Hence $g(k) > (2-o(1)) k$ for infinitely many primes $k$.
\end{corollary}

\begin{proof}
Let $p_n$ denote the $n$'th prime, and let $g_n=p_{n+1}-p_n$. The Prime Number Theorem implies that for any $\epsilon > 0$ there exists a natural number $N$ such that for any $n \geq N$, $g_n < \epsilon p_n$.
Now choose $n,k$ such that $N < p_n < 2^k$, and furthermore, $p_n$ is the largest prime $< 2^k$. Then
\begin{equation*}
p_{n+1} < 2^k + g_n < (1+\epsilon) 2^k.
\end{equation*}
Put $p=p_{n+1}$. If $q$ is any odd number $\geq p$ then, by Theorem~\ref{thm:prime}, $G_{q,2^k}$ is $\Z_p$-connected, but not $\Z_2^{k+1}$-connected.
It follows that $g(p) > 2(1-\epsilon)p$.
\end{proof}

There is a slight inaccuracy in the proof above, namely when $p=p_{n+1}=2^k+1$, that is, $p$ is a Fermat prime. But this can happen only if $k$ is a power of 2 which does not affect the correctness of Corollary~\ref{cor:infinite}.

\section{The cyclicity of a graph}

\begin{definition}
Let $G$ be a 2-edge-connected graph. We say that two edges $e_1,e_2$ are {\em cycle-equivalent} if every cycle that contains one of $e_1,e_2$ also contains the other. It is easy to see that this defines an equivalence relation on $E(G)$ and that two distinct edges $e_1,e_2$ are cycle-equivalent if and only if the two edges form a {\em 2-edge-cut}, that is, $G-e_1-e_2$ is disconnected.
We define the \textit{cyclicity} of $G$, denoted $q(G)$, to be the size of a largest equivalence class. In particular, if $G$ has no 2-edge-cuts (i.e. $G$ is 3-edge-connected), then $q(G) = 1$.
\end{definition}

The following result follows from Proposition 3.2 and Lemma 3.3 in \cite{l}. For the sake of completeness we include a proof.
 
\begin{proposition} \label{prop:q(G)}
If $G$ is $\Gamma$-connected, then $|\Gamma| > q(G)$.
\end{proposition}

\begin{proof}
Suppose $G$ is $\Gamma$-connected. Let $e_1, \ldots, e_{q(G)}$ be the edges in a largest cycle-equivalence class. Let $C$ be a cycle which contains one, and hence all of $e_1, \ldots, e_{q(G)}$. By reversing directions, if necessary, we may assume that all edges in $e_1, \ldots, e_{q(G)}$ have the same direction when we traverse $C$. For any flow, all edges in $e_1, \ldots, e_{q(G)}$ have the same flow value. We now use Theorem \ref{thm:lower}.
If $|\Gamma| \leq q(G)$ we can define $\varphi: E(G) \to \Gamma$ such that it is surjective on $e_1, \ldots, e_{q(G)}$, that is, all elements in $\Gamma$ are forbidden on the edges $e_1, \ldots, e_{q(G)}$. Hence $|\Gamma| > q(G)$.
\end{proof}

We shall prove the following (see Theorem~\ref{thm:2-edge-conn}).

\begin{theorem} \label{thm:2-edge-conn-main}
If $G$ is 2-edge-connected, then $G$ is $\Gamma$-connected for any Abelian group $\Gamma$ of order $|\Gamma| > 8q(G)^3$.
\end{theorem}

We can now combine Proposition~\ref{prop:q(G)} and Theorem~\ref{thm:2-edge-conn-main} to get:

\begin{corollary}
If $G$ is $\Gamma$-connected for some Abelian group $\Gamma$, then $G$ is $\Gamma'$-connected for any Abelian group $\Gamma'$ of order $|\Gamma'| > 8 |\Gamma|^3$.
\end{corollary}

\section{Flows in 3-edge-connected graphs with multiple forbidden flow-values}

We shall use the following definition and theorems by Jaeger et al. \cite{jlpt}:

\begin{definition} \label{defi:2-constructible}
A \textit{2-constructible} graph $G$ is defined recursively as follows.

(i) The graph with one vertex (and no edges) is 2-constructible.

(ii) If $G_1,G_2$ are 2-constructible, then the disjoint union of $G_1$ and $G_2$ together with two new edges joining them is 2-constructible.
\end{definition}

Jaeger et al. \cite{jlpt} proved the following.

\begin{theorem} \label{thm:specialspanningtree}
Let $G$ be a cubic 3-edge-connected graph and let $v$ be a vertex in $G$. Define $H = G - v$.
Then $H$ has a spanning tree $T$ such that the contraction of the edges of $H$ which are not in $T$ yields a 2-constructible graph.
\end{theorem}

They used it to prove the following.

\begin{theorem} \label{thm:jaeger}
Let $G$ be a 3-edge-connected graph and let $v$ be a vertex of degree 3 in $G$.
Then $G-v$ is $\Gamma$-connected for any Abelian group $\Gamma$ of order $|\Gamma| \geq 6$.
\end{theorem}

We shall also use the following definition and lemma.

\begin{definition}
Given a finite subset $\Pi = \{a_1, a_2, \ldots, a_k\}$ of an Abelian group $\Gamma$, we define the {\em simple sum} $\Pi'$ of $\Pi$ to be the set of all elements on the form
\begin{equation*}
\alpha_1 a_1 + \alpha_2 a_2 + \ldots + \alpha_k a_k
\end{equation*}
where $\alpha_i \in \{0, \pm 1\}$ for $1 \leq i \leq k$.
In particular, $\Pi'$ contains $0$.
\end{definition}

\begin{lemma} \label{lemma:subgroup}
Given a natural number $k$ and an Abelian group $\Gamma$ of order $|\Gamma| > k$, there exists a subset $\Pi \subseteq \Gamma$ which is closed under taking inverses and has $|\Pi| = k$ such that the simple sum $\Pi'$ of $\Pi$ satisfies $|\Pi'| \leq k^2$.
\end{lemma}

\begin{proof}
For $k=1$ we let $\Pi$ consist of $0$, so let $k>1$. If some element $a$ in $\Gamma$ has order at least $k$, then we put $\Pi = \{\pm a, \pm 2a, \ldots, \pm \frac{k}{2} a \}$ if $k$ is even, and $\Pi = \{0, \pm a, \pm 2a, \ldots, \pm \frac{k-1}{2} a \}$ if $k$ is odd. So assume all elements have order $<k$.

Let $\Gamma_0$ be the largest subgroup of order $<k$. Let $a \in \Gamma \setminus \Gamma_0$. Then the subgroup generated by $\Gamma_0 \cup \{a\}$ has order greater than $k$ but less than $k^2$ and any $k$-subset closed under taking inverses can play the role of $\Pi$.
\end{proof}

We use these results to prove the following:

\begin{theorem} \label{thm:3-edge-connminusvertex}
Let $G$ be a 3-edge-connected graph and let $v$ be a vertex of degree 3 in $G$. Define $H = G - v$. Assume each edge of $H$ has a direction.
Let $k \in \N$, and let $\Gamma$ be any Abelian group of order $|\Gamma| > 8k^3$. Assume that for each edge $e$, $F_e$ is a set of at most $k$ elements in $\Gamma$.
Then there exists a flow $f: E(H) \to \Gamma$ such that $f(e) \not\in F_e$ for all $e \in E(G)$.
\end{theorem}

We say that $F_e$ is the set of {\em forbidden flow values} for $e$.
If it becomes convenient to reverse the orientation of an edge $e$ we replace $F_e$ by $-F_e = \{-\gamma \mid \gamma \in F_e\}$.

\begin{proof}
By Lemma~\ref{lemma:subgroup} there exist subsets $\Pi, \Pi' \subseteq \Gamma$ such that $\Pi$ is closed under inverses and  has $|\Pi| = 2k$ and $\Pi'$ is the simple sum of $\Pi$ and has $|\Pi'| \leq 4k^2$.

It suffices to prove Theorem~\ref{thm:3-edge-connminusvertex} in the case where $G$ is cubic. For if $u$ is a vertex of degree $d>3$, then we replace $u$ by a cycle of length $d$ such that all vertices in that cycle have degree 3 in the resulting graph. (If $u$ is a cutvertex, we can make sure that no two edges of the cycle form a 2-edge-cut.) So assume that $G$ is cubic and 3-edge-connected. 
The case $k=1$ follows from Theorem~\ref{thm:lower} and Theorem~\ref{thm:jaeger}, so we may assume that $k>1$.

By Theorem~\ref{thm:specialspanningtree}, $H = G - v$ has a spanning tree $T$ in $H$ such that the contraction of the edges of $H$ which are not in $T$ yields a 2-constructible graph.
We colour the edges of $T$ red and colour the edges in $H-T$ blue. Let $E_b$ be the set of blue edges.

As $H/E_b$ is 2-constructible there exists a sequence $H_0, H_1, \ldots, H_t$ of graphs such that $H_0$ is the empty graph on $|V(H/E_b)|$ vertices, $H_t = H/E_b$, and each $H_i$ is obtained from $H_{i-1}$ by adding two red edges (which we denote by $e_i$ and $e_i'$) between two disjoint connected components of $H_{i-1}$.

We first describe informally the method, which is similar to the method in \cite{jlpt}. We give $H$ flow in two steps. In the first step we give all red edges flow values which are non-forbidden and which remain to be non-forbidden after Step 2. Step 1 also affects the blue edges but this is not important. In Step 2 we send flow through the blue edges such that the flow values of the blue edges are non-forbidden. There is already flow through the blue edges after Step 1. The additional flow values (added in Step 2) in the blue edges will be in $\Pi$. This will then affect the red edges in such a way that the additional flow through a red edge will be in $\Pi'$. In Step 1, the flow is chosen such that an additional flow value in $\Pi'$ will keep the flow value in the red edge non-forbidden.

We now argue formally.

\vspace*{3mm}

\textbf{Step 1:}
Consider a cycle $C_t'$ in $H_t=H/E_b$ through $e_t$ and $e_t'$ and let $C_t$ be a cycle in $H$ containg the edges of $C_t'$ and no other red edges.
If the orientations of $e_t$ and $e_t'$ do not agree on $C_t$ then reverse the orientation of $e_t$ and replace $F_e$ by $-F_e$.
Let $F=F_{e_t} \cup F_{{e'}_t}$.
We pick flow $\gamma_t$ to send through $C_t$ such that:
\begin{align}
(\gamma_t + \Pi') \cap F = \emptyset.
\end{align}
Since $\Pi'$ has size at most $4k^2$, $F$ has size at most $2k$, and $\gamma_t$ can be chosen in more than $8k^3$ ways, this is indeed possible.
Modify the set of forbidden flow values in each red edge $e$ of $C_t - e_t - e'_t$ such that the new set of forbidden flow values is $F_e - \gamma_t = \{\alpha - \gamma_t \mid \alpha \in F_t\}$ if the orientation of $e$ agrees with that of $e_t$, and $F_e + \gamma_t = \{\alpha + \gamma_t \mid \alpha \in F_t\}$ if not. We call this the {\em first iteration} of Step 1.

Next, we consider a cycle $C'_{t-1}$ in $H/E_b-e_t-e'_t$ through $e_{t-1}$ and $e_{t-1}'$ and let $C_{t-1}$ be a cycle in $H-e_t-e'_t$ containing the edges of $C'_{t-1}$ and no other red edges. As above we find an appropriate flow $\gamma_{t-1}$ to send though $C_{t-1}$. Note that $e_t$ and $e_t'$ are not in $C_{t-1}$ (or any other $C_i$ for $1\leq i \leq t-1$) by the construction of $H/E_b$, so the flow in $e_t,e_t'$ will not be changed in Step 1. We call this the {\em second iteration} of Step 1.
We repeat this argument for $e_{t-2}$ and $e_{t-2}'$ and then $e_{t-3}$ and $e_{t-3}'$, etc., until all red edges have received a flow value. Once $e_i,e_i'$ have received a flow value in the $(t+1-i)$'th iteration of Step 1, that flow value will not be further changed in Step 1.

\vspace*{3mm}

\textbf{Step 2:}
If all blue edges have a non-forbidden flow after Step 1, the proof is complete and there will be no Step 2. So consider a blue edge $e$ which has a forbidden flow value $f'(e)$, say, after Step 1. 
Since $|\Pi| = 2k$ and $F_e$ has at most $k$ elements, there exists a $\gamma \in \Pi$ such that $f'(e)+\gamma$ and $f'(e)-\gamma$ are both non-forbidden, that is, they are both in $\Gamma \setminus F_e$. We say that $\gamma$ is {\em good for $e$}.
Let $E_{\gamma}$ be the set of all edges which currently have forbidden flow values and for which $\gamma$ is good.
For each edge in $E_{\gamma}$, let $C_e$ be the unique cycle in $T+e$.
Form the symmetric difference $H_{\gamma}$ of $C_e$ taken over all edges $e$ in $E_{\gamma}$.
Then $H_{\gamma}$ is an even graph (that is, a graph where each component is Eulerian), and hence $H_{\gamma}$ has a flow using only $\gamma$ and $-\gamma$. We add this flow to the flow obtained after Step 1, and now all edges in $E_{\gamma}$ have non-forbidden flow values.

Repeat this step as long as there are blue edges with forbidden flow values.

\vspace*{3mm}

Consider now a red edge $e$ after Step 2. Let $f'(e)$ be its flow value after Step 1. In Step 2 we add some elements in $\Pi$ to $f'(e)$. Thus the final flow of $e$ is of the form $f'(e)+ \gamma'$ where $\gamma'$ is a simple sum of elements from $\Pi$, that is $\gamma' \in \Pi'$. By the choice of $f'(e)$ it follows that the final flow value $f'(e) + \gamma' \not \in F_e$, as required. (In Step 1 we modified the forbidden flow values in the red edges. Here in Step 2, $F_e$ denotes the original forbidden flow values.)
\end{proof}

Theorem~\ref{thm:3-edge-conn} below follows from Theorem~\ref{thm:3-edge-connminusvertex} by adding a vertex of degree 3 to $G$ which may be removed again.

\begin{corollary} \label{thm:3-edge-conn}
Let $G$ be a 3-edge-connected graph.
Let $k \in \N$, and let $\Gamma$ be any Abelian group of order $|\Gamma| > 8k^3$.
Given any orientation of $G$, if $G$ has at most $k$ forbidden flow values from $\Gamma$ in each edge, then there exists a flow $f: E(G) \to \Gamma$ such that $f(e)$ is not forbidden for any $e \in E(G)$.
\end{corollary}

\section{Flows in 2-edge-connected graphs with forbidden flow-values}

Now we can use Theorem~\ref{thm:3-edge-conn} to prove a similar statement about 2-edge-connected graphs.

\begin{theorem} \label{thm:2-edge-conn}
Let $G$ be a 2-edge-connected graph.
Let $k \in \N$, and let $\Gamma$ be an Abelian group of order $|\Gamma| > 8(k q(G))^3$.
Given any orientation of $G$, if $G$ has at most $k$ forbidden flow values from $\Gamma$ in each edge, then there exists a flow $f: E(G) \to \Gamma$ such that $f(e)$ is not forbidden for any $e \in E(G)$.
\end{theorem}

\begin{proof}
We may assume that the orientation of $G$ is strongly connected by reversing the edges necessary and adjusting the forbidden sets in the edges accordingly. 
Then all edges in a cycle-equivalence class have the same direction in each cycle containing them. We form a new directed graph $G'$ on the same edge set as $G$ (but possibly with a different vertex set) such that

(i): the edge set of any cycle in $G$ is also the edge set of a cycle in $G'$ and vice versa (in particular $G$ and $G'$ have the same cycle-equivalence classes).

(ii): the orientations of all edges agree in any two cycles with the same edge set in $G$ and $G'$, respectively. More precisely: If $E$ is the edge set of a cycle in $G$ and hence also a cycle in $G'$, then we can choose a clockwise orientation of $C,C'$ such that an edge $e$ is directed clockwise in $C$ if and only if $e$ is directed clockwise in $C'$.

(iii): every $\Gamma$-flow in $G$ is also a $\Gamma$-flow in $G'$ and vice versa.

(iv): the edges of each cycle-equivalence class in $G'$ form a directed path such that each intermediate vertex has indegree 1 and outdegree 1.

\vspace{3mm}

Note that (iv) is equivalent with the following:

(v): If we delete the edges of a cycle-equivalence class in $G'$, then the resulting graph has precisely one component with edges.

\vspace{3mm}

If $G$ satisfies (iv), we put $G'=G$. Otherwise, there exists a cycle-equivalence class such that the deletion of its edges results in a graph with more than one component. Let $G_1$ be one component. Let $G_2$ consist of all other components containing edges together with those paths in the cycle-equivalence class that connect them. Then the edges in the cycle-equivalence class that are not in $G_1 \cup G_2$ form two directed paths $P_1$ with vertices $v_1,v_2, \ldots ,v_s$ and directed edges $e_1,e_2, \ldots e_{s-1}$ and
$P_2$ with vertices $u_1,u_2, \ldots ,u_t$ and directed edges  $e'_1,e'_2, \ldots e'_{t-1}$ where $v_1,u_t \in V(G_1)$ and  $u_1,v_s \in V(G_2)$. 
The two paths $P_1,P_2$ are disjoint except that possibly $v_1=u_t$ and possibly $u_1=v_s$. Only their endvertices are in $G_1 \cup G_2$. Now we form a new graph $H$ from $G_1 \cup G_2$ by first identifying $v_1,v_s$ and then adding a directed path with edges $e_1,e_2, \ldots e_{s-1},e'_1,e'_2, \ldots e'_{t-1}$ from $u_1$ to $u_t$. Then $H$ also satisfies (i),(ii),(iii). Moreover, $H$ has more vertices of indegree 1 and outdegree 1 than $G$. So, in a finite number of steps we obtain a graph satisfying (i),(ii),(iii),(iv).

It follows that $G'$ is a subdivision of a 3-edge-connected graph $G''$. An edge in $G''$ corresponds to a cycle-equivalence class in $G$ and is therefore subdivided into at most $q(G)$ edges, by the definition of $q(G)$.
Now we complete the proof by applying  Theorem~\ref{thm:3-edge-conn} to $G''$. As each edge in $G$ has at most $k$ forbidden flow values, each edge in $G''$ has at most $kq(G)$ forbidden flow values.
\end{proof}

Theorem~\ref{thm:2-edge-conn} for $k=1$ combined with Theorem~\ref{thm:lower} implies Theorem~\ref{thm:2-edge-conn-main} as well as the upper bound $g(k) \leq 8k^3+1$.

\section{Group coloring}

Jaeger et al. \cite{jlpt} define group colourability as follows.

\begin{definition}
Let $\Gamma$ be an Abelian group. The graph $G$ is said to be {\em $\Gamma$-colorable} if the following holds: Given some orientation $D$ of $G$ and any function $\varphi: E(G) \to \Gamma$ there exists a vertex coloring $c: V(G) \to \Gamma$ such that
$c(w) - c(u) \neq \varphi(uw)$
for each $uw \in E(D)$.
\end{definition}

We say that $\varphi$ \textit{allows} $c$.

\vspace*{3mm}

A graph is {\em $d$-degenerate} if every subgraph contains a vertex of degree at most $d$.
The {\em coloring number}  $Col(G)$ is the largest number such that $G$ is $(d-1)$-degenerate. Equivalently, $Col(G) -1$ is the maximum minimum degree where the maximum is taken over all subgraphs of $G$.
Clearly

\begin{align}
\chi(G) \leq \chi_{wg}(G) \leq \chi_g(G) \leq Col(G).
\end{align}

Also

\begin{align}
 \chi_l(G) \leq Col(G),
\end{align}

where $\chi_l(G)$ is the list-chromatic number.

\vspace{3mm}

Now let $G$ be a graph with $\chi_{wg}(G) = k$, and let $\Gamma$ be an Abelian group such that $G$ is $\Gamma$-colorable.
Kr\'al' et al. \cite{kpv} proved that $k> \delta/2\ln(\delta)$ where $\delta$ is the minimum degree of $G$. (Kr\'al' et al. formulated their result as one about the group chromatic number but the proof works for the weak group chromatic number as well.)
As this holds for every subgraph of $G$ it also holds for a subgraph of minimum degree $\delta = Col(G) -1$.
Hence 

\begin{align}
\chi_{wg}(G) = k > \frac{\delta}{2 \ln \delta}
= \frac{Col(G) - 1}{2 \ln(Col(G)-1)} 
\geq \frac{\chi_g(G) - 1}{2 \ln(\chi_g(G)-1)}.
\end{align}

This implies, for each natural number $k$, the upper bound in the following:

\begin{align}
(2-o(1)) k < h(k) < (2+o(1))k \ln(k).
\end{align}

Since the graphs $G_{q,k}$ are planar, their dual graphs establish the lower bound (for infinitely many primes $k$) by the proof of Theorem~\ref{thm:prime}.

The group chromatic number and weak group chromatic number have some similarity to the list-chromatic number. Indeed, the proof of the 5-list-color theorem for planar graphs \cite{t1} translates, word for word, to the result that every planar graph has group chromatic number at most 5. And the proof of the 3-list-color theorem for planar graphs of girth at least 5 \cite{t3} translates to the result that every planar graph has group chromatic number at most 3. Only a minor detail in the proof needs additional explanation, see \cite{rty}.

\vspace{3mm}

The following conjecture was made by Kr\'al' et al. in \cite{kpv} and, according to \cite{kpv}, independently by Margit Voigt.

\begin{conjecture} \label{prob:1}
For every graph $G$, $\chi_l(G) \leq \chi_g(G)$.
\end{conjecture}

We propose the analogous conjecture for the weak group chromatic number.

\begin{conjecture} \label{prob:2}
For every graph $G$, $\chi_l(G) \leq \chi_{wg}(G)$.
\end{conjecture}

In \cite{jlpt} it is shown that every graph with two edge-disjoint spanning trees is $\Gamma$-connected for every Abelian group $\Gamma$ of order at least 4. This implies that $g(3)=3$ as pointed out in \cite{li}. 
For, if $G$ is $\Z_3$-connected we let $\beta(v)=1-d(v)$ (where $d(v)$ denotes the degree of $v$) for every vertex $v$, except possibly one, in Definition~\ref{def:delta}. 
We may choose the resulting $f$ such that $f(e)=1$ for every edge $e$, by reversing the direction of those edges having flow value 2. 
This gives an orientation of the edges such that all vertices, except possibly one, have outdegree $2 \pmod{3}$ and hence outdegree at least 2. Thus $G$ has at least $2|V(G)|-2$ edges. 
As this also holds for every graph obtained from $G$ by identifying vertices (and removing the loops that may arise), $G$ has two edge-disjoint spanning trees, by a fundamental result of Edmonds \cite{e}, Nash-Williams \cite{n} and Tutte \cite{t}.

Note that the dual statement does not hold: A $\Z_3$-colorable graph is not necessarily the union of two spanning trees. For example is $K_{3,5}$ $\Z_3$-colorable (see e.g. \cite{la}) but contains 8 vertices and 15 edges, thus it has too many edges to be the union of two spanning trees. 

\begin{conjecture} \label{prob:4}
$h(3)=3$.
\end{conjecture}

Conecture~\ref{prob:4} clearly implies Conjecture~\ref{prob:5} below.

\begin{conjecture} \label{prob:5}
If $G$ is $\Z_3$-colorable, then $G$ is $\Z_4$-colorable and also $\Z_2 \times \Z_2$-colorable.
\end{conjecture}

Since $g(3)=3$, the answer is affirmative for planar graphs. 
Jaeger et al. \cite{jlpt} asked if a graph is $\Z_4$-connected if and only if it is $\Z_2 \times \Z_2$-connected. Hu\v{s}ek et al. \cite{h} answered this in the negative for both implications. As some of the counterexamples are planar, their dual graphs show that $\Z_4$-colorability does not imply $\Z_2 \times \Z_2$-colorability. These graphs have multiple edges. Examples without multiple edges can be obtained using Hajos' construction. Such examples can even be made planar. Their dual graphs are therefore 3-edge-connected planar graphs that are $\Z_4$-connected, but not $\Z_2 \times \Z_2$-connected. For details, see \cite{lan}.
We do not know if $\Z_4$-colorability is implied by $\Z_2 \times \Z_2$-colorablility. 

\vspace*{3mm}

We conclude the paper by showing that Conjecture~\ref{prob:5} in fact is equivalent to Conjecture~\ref{prob:4}. This follows from the two propositions below.

\begin{proposition} \label{prop:1}
If $G$ is $\Z_3$-colorable, then $G$ is 5-degenerate, that is, $\chi_g(G) \leq Col(G) \leq 6$.
\end{proposition}

\begin{proof}
Consider the bipartite graph with vertices $A \cup B$ in which the vertices in $A$ corresponds to all possible $\Z_3$-colorings of the vertices in $G$, and the vertices in $B$ correspond to all possible $\Z_3$-colorings of the edges in $G$. 
We join two vertices by an edge if the edge-coloring allows the vertex-coloring.

Let $n = |V(G)|$ and $m = |E(G)|$.
Note that there $3^n$ vertices in $A$ and $3^m$ vertices in $B$. Furthermore, each vertex in $A$ has degree $2^m$.
As each edge-coloring will have at least one allowed vertex-coloring, all vertices in $B$ has degree at least 1. 
Thus $3^n \cdot 2^m \geq 3^m$. We conclude 
\begin{align}
m \leq n \cdot \frac{\log(3)}{\log(\frac{3}{2})} < 2.8 n < 3n,
\end{align} 
so $G$ is 5-degenerate. Thus, $G$ is $\Gamma$-colorable for every Abelian group $\Gamma$ of order $|\Gamma| \geq 6$.
\end{proof}

We do not know if every $\Z_3$-colorable graph is even 3-degenerate.

\begin{proposition} \label{prop:2}
If $G$ is $\Z_3$-colorable, then $G$ is $\Z_5$-colorable.
\end{proposition}

\begin{proof}
Let $G$ be a $\Z_3$-colorable graph, let $D$ be an orientation of $G$, and let $\varphi: E(D) \to \Z_5$ be given. We think of $\Z_5$ as the numbers $0,1,-1,2,-2$.
We define $\varphi': E(D) \to \Z_3$ by reducing these numbers modulo 3, where we think of $\Z_3$ as the numbers $0,1,-1$. More precisely, for each edge $e \in E(D)$,
\begin{align}
\varphi'(e) =
\begin{cases}
0 & \text{if } \varphi(e) = 0, \\
1 & \text{if } \varphi(e) = 1 \text{ or } \varphi(e) = -2, \\
-1 & \text{if } \varphi(e) = 2 \text{ or } \varphi(e) = -1.
\end{cases}
\end{align}
As $G$ is $\Z_3$-colorable, there exists a vertex-coloring $c': V(G) \to \Z_3$ such that $c'(v) - c'(u) \neq \varphi'(uv)$ for each $uv \in E(D)$.
Consider the vertex-coloring $c: V(G) \to \Z_5$ defined by $c(v) = c'(v) \pmod{5}$ for each vertex $v$.
We argue that $c$ is a proper coloring. Let $e = uv$ be a directed edge. 
If $c'(v) - c'(u) = 0$, then $c(v) - c(u) = 0$. As $\varphi'(e) \neq 0$ we get $\varphi(e) \neq 0$ so $c(v) - c(u) \neq \varphi(e)$.
If $c'(v) - c'(u) = 1$, then $c(v) - c(u) = 1$ or $c(v) - c(u) = -2$. As $\varphi'(e) \neq 1$ we get $\varphi(e) \neq 1$ and $\varphi(e) \neq -2$ so $c(v) - c(u) \neq \varphi(e)$.
If $c'(v) - c'(u) = -1$, then $c(v) - c(u) = 2$ or $c(v) - c(u) = -1$. As $\varphi'(e) \neq -1$ we get $\varphi(e) \neq 2$ and $\varphi(e) \neq -1$ so $c(v) - c(u) \neq \varphi(e)$.
\end{proof}

\end{document}